\documentclass[a4paper,11pt,oneside]{article}

\usepackage{amsmath,amssymb,mathtools}

\usepackage[english,francais]{babel}

\newtheorem{theorem}{Theorem}[section]
\newtheorem{lemma}[theorem]{Lemma}
\newtheorem{e-proposition}[theorem]{Proposition}
\newtheorem{corollary}[theorem]{Corollary}
\newtheorem{e-definition}[theorem]{Definition\rm}
\newtheorem{remark}{\it Remark\/}

\newenvironment{proof}{\paragraph*{\textbf{Proof}}}{\hfill$\square$}

\def\og{\leavevmode\raise.3ex\hbox{$\scriptscriptstyle\langle\!\langle$~}}
\def\fg{\leavevmode\raise.3ex\hbox{~$\!\scriptscriptstyle\,\rangle\!\rangle$}}

\newcommand{\R}{\mathbb{R}}

\newcommand{\calN}{\mathcal{N}}

\title{Variants of the Empirical Interpolation Method: symmetric formulation, choice of norms and rectangular extension}

\author{Fabien Casenave$^{1,2}$, Alexandre Ern$^{3}$ and Tony Leli\`evre$^{3}$\\~\\
$^1$ IGN LAREG, Univ Paris Diderot, Sorbonne Paris Cit\'e,\\5 rue Thomas Mann, 75205 Paris Cedex 13, France\\
$^2$ currently at SafranTech, Rue des Jeunes Bois,\\Ch\^ateaufort, CS 80112, 78772 Magny-Les-Hameaux, France\\
$^3$ Universit\'e Paris-Est, CERMICS (ENPC),\\ 77455 Marne la Vall\'ee Cedex 2, France}

\date{}

\begin{document}

\maketitle

\section*{Abstract}
The Empirical Interpolation Method (EIM) is a greedy procedure that constructs approximate representations of two-variable functions in separated form.
In its classical presentation, the two variables play a non-symmetric role.
In this work, we give an equivalent definition of the EIM approximation, in which the two variables play symmetric roles.
Then, we give a proof for the existence of this approximation, and extend it up to the convergence of the EIM,
and for any norm chosen to compute the error in the greedy step.
Finally, we introduce a way to compute a separated representation in the case where the number of selected values is different for each variable.
In the case of a physical field measured by sensors, this is useful to discard a broken sensor while keeping the information provided by the associated selected field.

\noindent{\bf MSC2010 classification} \vskip 0.5\baselineskip \noindent
{
65D05, 65D15, 68W25
}

\noindent{\bf Keywords} \vskip 0.5\baselineskip \noindent
{
Empirical Interpolation Method, rectangular case, sensor failure, symmetry
}

\selectlanguage{english}

\section{Introduction}

Consider a function $f: {\mathcal X} \times {\mathcal Y} \to\R$.
The Empirical Interpolation Method (EIM)~\cite{Barrault,Maday} is an offline/online procedure that provides an approximate representation $I_d(f)$ of $f$ in separated form, where
the integer $d$ denotes the number of terms in the representation.
The offline stage of the EIM consists in selecting some points $(x_1, \ldots, x_d) \in {\mathcal X}^d$
and $(y_1, \ldots, y_d) \in {\mathcal Y}^d$ in a greedy fashion, such that
\begin{subequations}
\begin{alignat}{1}
x_{k+1}&=\arg\max_{x \in {\mathcal X}} \|\left(f-I_k(f)\right)(x, \cdot)\|_{L^{\infty}(\mathcal{Y})},\label{eq:points_select1}\\
y_{k+1}&=\arg\max_{y \in  {\mathcal Y}} |\left(f-I_k(f)\right)(x_{k+1},y)|,\label{eq:points_select2}
\end{alignat}
\end{subequations}
where $I_k(f)$ denotes the
separated representation constructed with the $k$ first points $(x_l,y_l)$, $l\in\{1,..., k\}$.
The method is efficient when the error $f-I_d(f)$ is small for reasonably small values of $d$.
Some functions $q_l(y)$, $l\in\{1,..., d\}$ and a matrix $B$ of size $d\times d$ depending on these
points are also constructed, see~\cite{Barrault} and~\cite{Maday} for details.
The separated representation is then obtained as 
\begin{equation}
\label{eq:classical_sep_rep}
I_d(f)(x,y)=\sum_{1\le l\le d} \lambda_l(x)q_l(y), 
\end{equation}
where
the functions $\lambda_l(x)$, $l\in\{1,..., d\}$, solve the linear system
$\sum_{m=1}^{d}B_{l,m}{\lambda}_m(x)=f(x,y_l)$, $l\in\{1,..., d\}$.
The function $I_d(f)$ satisfies the following interpolation property~{\cite[Lemma~1]{Maday}}:
for all $m\in\{1,..., d\}$,
\begin{subequations}
\begin{alignat}{1}
I_d (f)(x,y_m) &= f(x,y_m), \quad \textnormal{for all } x\in{\mathcal X},\label{eq:interpolation_prop1}\\
I_d (f)(x_m,y) &= f(x_m,y), \quad \textnormal{for all } y\in{\mathcal Y}.\label{eq:interpolation_prop2}
\end{alignat}
\end{subequations}
In practice, the size $d$ is not chosen a priori, and
the greedy procedure stops when $|\left(f-I_k(f)\right)(x_{k+1},y_{k+1})|$ is small enough.
Define $\mathcal{U}:=\{f(x,\cdot), x\in\mathcal{X}\}$. Elements of $\mathcal{U}$ are functions from $\mathcal{Y}$ to $\mathbb{R}$.
\begin{theorem}[Existence of the decomposition,~{\cite[Theorem~1]{Maday}}]
\label{existence}
Assume that the interpolation points are chosen according to~\eqref{eq:points_select1}-\eqref{eq:points_select2}
and that $d < \dim{\rm span}(\mathcal{U})$. Then, the separated representation~\eqref{eq:classical_sep_rep} is well-defined.
\end{theorem}

In~\cite{casenave_ACM}, it is observed that $I_d(f)$ can be rewritten in an algebraically equivalent form as
\begin{equation}
\label{eq:sym_form}
I_d(f)(x,y)=\sum_{1\leq l,m\leq d} D_{l,m}f(x_l,y)f(x,y_m), 
\end{equation}
where the matrix $D$ depends on the points $x_l$, $y_m$, $l,m\in\{1,...,d\}$, and can be constructed recursively during the offline stage of the EIM.
It is easy to check that~\eqref{eq:interpolation_prop1}-\eqref{eq:interpolation_prop2} is satisfied if and only if $D=F^{-T}$, where $F_{l,m} = f(x_l,y_m)$,
which motivates an alternative presentation of the EIM based on Equation~\eqref{eq:sym_form}
where the variables $x$ and $y$ play symmetric roles. The double summation in~\eqref{eq:sym_form} emphasizes this symmetric role; note that, for instance, 
$I_d (f)(x,y)=\sum_{1\le l\le d} \tilde{\lambda}_l(x)\tilde{q}_l(y)$, with $\tilde{\lambda}_l(x)=\sum_{1\le m\le d} D_{l,m}f(x,y_m)$ and $\tilde{q}_l(y) = f(x_l,y)$.

Let $\|{\cdot}\|_{\mathcal Y}$ be a norm on~${\mathcal Y}$ and suppose that the interpolation points are now selected as
\begin{subequations}
\begin{alignat}{1}
x_{k+1}&=\arg\max_{x \in {\mathcal X}} \|\left(f-I_k(f)\right)(x, \cdot)\|_{\mathcal Y},\label{eq:MP1}\\
y_{k+1}&=\arg\max_{y \in  {\mathcal Y}} |\left(f-I_k(f)\right)(x_{k+1},y)|,\label{eq:MP2}
\end{alignat}
\end{subequations}
the difference with~\eqref{eq:points_select1}-\eqref{eq:points_select2} being the arbitrary choice for the norm $\|{\cdot}\|_{\mathcal Y}$ in the first line,
instead of just $\|{\cdot}\|_{L^{\infty}(\mathcal{Y})}$.
One can exchange the roles of $x$ and $y$ in the previous algorithm, leading to
\begin{subequations}
\begin{alignat}{1}
y_{k+1}&=\arg\max_{y \in {\mathcal Y}} \|\left(f-I_k(f)\right)( \cdot, y)\|_{\mathcal X},\label{eq:MP'1}\\
x_{k+1}&=\arg\max_{x \in  {\mathcal X}} |\left(f-I_k(f)\right)(x,y_{k+1})|,\label{eq:MP'2}
\end{alignat}
\end{subequations}
for an arbitrary norm $\|{\cdot}\|_{\mathcal X}$ on ${\mathcal X}$. In general, the couple $(x_{k+1},y_{k+1})$ resulting from~\eqref{eq:MP'1}-\eqref{eq:MP'2} differs from the one obtained
with~\eqref{eq:MP1}-\eqref{eq:MP2}.
Choosing the $L^{\infty}$-norm in ${\mathcal Y}$ and ${\mathcal X}$ in~\eqref{eq:MP1} and~\eqref{eq:MP'1} respectively, we infer that
\begin{equation*}
\label{eq:MP_Linfty}
(x_{k+1},y_{k+1})=\arg\max_{(x,y) \in  {\mathcal X} \times {\mathcal Y} } |\left(f-I_k(f)\right)(x,y)|,
\end{equation*}
thus recovering the choice made in~\eqref{eq:points_select1}-\eqref{eq:points_select2}.
For this specific choice of $L^\infty$-norms on $\mathcal X$ and $\mathcal Y$,
\eqref{eq:MP1}-\eqref{eq:MP2} is actually equivalent to~\eqref{eq:MP'1}-\eqref{eq:MP'2}.

The first contribution of this work is the following result:
\begin{theorem}
\label{existence2}
Assume that the interpolation points are chosen according to~~\eqref{eq:MP1}-\eqref{eq:MP2} or~\eqref{eq:MP'1}-\eqref{eq:MP'2} and
that $d \leq \dim{\rm span}(\mathcal{U})$. Then, the separated representation~\eqref{eq:sym_form} with $D=F^{-T}$, where $F_{l,m} = f(x_l,y_m)$,
is well-defined and satisfies~\eqref{eq:interpolation_prop1}-\eqref{eq:interpolation_prop2}.
\end{theorem}
Theorem~\ref{existence2} extends Theorem~\ref{existence} in two ways. First, it allows for a more general selection of the interpolation points.
Second, the existence of the decomposition is ensured up to convergence. Since $\dim{\rm span}(\mathcal{U})$ is usually unknown, a practical consequence is that
for all $k\in\mathbb{N}$, either $f(x,y)=I_k(f)(x,y)$ for all $(x,y)\in\mathcal{X}\times\mathcal{Y}$, or the greedy procedure can be pursued.
Moreover, we observe that the proof of Theorem~\ref{existence2} given below is straightforward and does not rely explicitly on the Kolmogorov $n$-witdh.

The second contribution of this work is that starting from~\eqref{eq:sym_form}, we devise a rectangular extension of EIM.
An interesting application of this extension is an improved recovery procedure if it is decided a posteriori that the values of $f$ at some points $y_l$ are not
reliable and should be discarded. This situation is motivated for instance by sensor failure in the context of the Generalized Empirical Interpolation Method (GEIM)~\cite{GEIM},
which we address at the end of this work. Note that discarding a posteriori some interpolation point $y_l$ is not straightforward in the standard EIM setting since the
whole construction relies on couples of points and functions defined by induction. In the present setting, the interpolation points $x_l$ can be kept even if the points $y_l$ are discarded, thereby leading
to a rectangular formulation. The key idea is to replace the matrix inversion by a pseudo-inverse to compute the rectangular matrix $D$ to be used in~\eqref{eq:sym_form}.

\section{Symmetric formulation of EIM and proof of Theorem~\ref{existence2}}
\label{sec:sym}

\subsection{Symmetric formulation}

Consider some interpolation points $(x_1, \ldots, x_d) \in {\mathcal X}^d$ and $(y_1, \ldots, y_d) \in {\mathcal Y}^d$, and
define the matrix $F\in\mathbb{R}^{d\times d}$ such that $F_{l,m}=f(x_l,y_m)$, for all $l,m\in\{1,...,d\}$.
\begin{lemma}[Existence]
\label{lemma_0}
If the matrix $F$ is invertible, there exists a function $I_d(f):{\mathcal X} \times {\mathcal Y} \to\R$ in the form~\eqref{eq:sym_form}
satisfying the interpolation property~\eqref{eq:interpolation_prop1}-\eqref{eq:interpolation_prop1}.
\end{lemma}
\begin{proof}
We observe that for all $m_0\in\{1,...,d\}$ and all $x\in\mathcal{X}$,
\begin{equation*}
I_d(f)(x,y_{m_0}) = \sum_{1 \le l,m \le d} D_{l,m} f(x_l,y_{m_0}) f(x,y_m) = \sum_{1 \le m \le d} \left[ \sum_{1 \le l \le d} D_{l,m} f(x_l,y_{m_0}) \right] f(x,y_m).
\end{equation*}
Hence, if 
$F^T D = {\rm Id}$,~\eqref{eq:interpolation_prop1} holds.
Similarly, if $D F^T = {\rm Id}$,~\eqref{eq:interpolation_prop2} holds.
Since $F$ is invertible, we see that we can choose $D= F^{-T}$.
\end{proof}

Note that the EIM interpolation is thus given by $I_d(f)(x,y)= \sum_{1 \le l,m \le d} (F^{-1})_{m,l} f(x_l,y) f(x,y_m)$.

\subsection{Proof of Theorem~\ref{existence2}}

We begin with the following result on the stopping criterion of the greedy procedure.
\begin{lemma}[Stopping criterion]
\label{lemma_1}
Consider a set of points $(x_l,y_m)_{1 \le l,m \le k+1}$ and assume that the matrix $F^k:=(f(x_l,y_m))_{1 \le l,m \le k}$ is
invertible.
If $$\left(f-I_k(f)\right)(x_{k+1},y_{k+1}) \neq 0,$$ then the matrix
$F^{k+1}:=(f(x_l,y_m))_{1 \le l,m \le k+1}$ is also invertible.
 \end{lemma}
\begin{proof}
Assume that $F^{k+1}$ is not invertible. This means that there exist $(\lambda_1, \ldots, \lambda_{k+1})$ not all zero such that,
$\forall l \in \{1, \ldots,k+1\}$, $\sum_{m=1}^{k+1} \lambda_m f(x_l,y_m)=0$, which writes
\begin{equation}
\label{eq:1}
\sum_{1\le m\le k} \lambda_m f(x_l,y_m)= - \lambda_{k+1} f(x_l,y_{k+1}).
\end{equation}
Note that necessarily, $\lambda_{k+1} \neq 0$, since otherwise the matrix $F^k$ would not be
invertible. Let us introduce the matrix $D^k=(F^k)^{-T}$ so that
$I_k(f)(x,y)=\sum_{1 \le l,m \le k} D^k_{l,m} f(x_l,y)f(x,y_m)$. Let $m_0\in\{1,...,k\}$. From~\eqref{eq:1}, we infer that
$$\sum_{1\le l\le k} D^k_{l,m_0} \sum_{1\le m\le k} \lambda_m f(x_l,y_m)= - \lambda_{k+1}\sum_{1\le l\le k} D^k_{l,m_0} f(x_l,y_{k+1}).$$
and since, for all $m,m_0 \in\{1,...,k\}$, $\sum_{1\le l\le k} D^k_{l,m_0}  f(x_l,y_m)= \delta_{m,m_0}$,
exchanging summations in the left-hand side leads to
$$ \lambda_{m_0}= - \lambda_{k+1}\sum_{1\le l\le k} D^k_{l,m_0} f(x_l,y_{k+1}).$$
We deduce that
$$\sum_{1\le m_0\le k} \lambda_{m_0} f(x,y_{m_0})=- \lambda_{k+1}
\sum_{1\le l\le k}~\sum_{1\le m_0\le k} D^k_{l,m_0} f(x_l,y_{k+1})
f(x,y_{m_0})=-\lambda_{k+1} I_k(f)(x,y_{k+1}).$$
Taking $x=x_{k+1}$ and using again~\eqref{eq:1}, we infer that
$$- \lambda_{k+1} f(x_{k+1},y_{k+1})=-\lambda_{k+1} I_k(f)(x_{k+1},y_{k+1})$$
and recalling that $\lambda_{k+1} \neq 0$, we conclude that
$$ f(x_{k+1},y_{k+1})= I_k(f)(x_{k+1},y_{k+1}).$$
\end{proof}

It is clear from~\eqref{eq:MP1}-\eqref{eq:MP2} or~\eqref{eq:MP'1}-\eqref{eq:MP'2} that for all $k<\dim{\rm span}(\mathcal{U})$, $\left(f-I_k(f)\right)(x_{k+1},y_{k+1}) \neq 0$. Then, from 
Lemma~\ref{lemma_1}, $I_{k+1}$ is well-defined, and Theorem~\ref{existence2} is proved.

\begin{corollary}[Termination]
\label{corollary_1}
If $(x_{k+1},y_{k+1})$ is selected according to~\eqref{eq:MP1}-\eqref{eq:MP2} or~\eqref{eq:MP'1}-\eqref{eq:MP'2}, then
$|\left(f-I_k(f)\right)(x_{k+1},y_{k+1})| = 0,$ implies that $\left(f-I_k(f)\right)(x,y)=0$ for
all $(x,y) \in {\mathcal X} \times {\mathcal Y}$. 
\end{corollary}
\begin{proof}
Suppose that $(x_{k+1},y_{k+1})$ is selected according to~\eqref{eq:MP1}-\eqref{eq:MP2}. Then
$|\left(f-I_k(f)\right)(x_{k+1},y_{k+1})|~=~0$ implies by definition of $y_{k+1}$ that for all $y \in {\mathcal Y}$,
$|\left(f-I_k(f)\right)(x_{k+1},y)|~=~0$. This in turn means that $\|\left(f-I_k(f)\right)(x_{k+1},\cdot)\|_{\mathcal Y}~=~0$. Then, by definition of
$x_{k+1}$, for all $x \in {\mathcal X}$, $\|\left(f-I_k(f)\right)(x,\cdot)\|_{\mathcal Y}=0$, and thus $f=I_k(f)$.
The proof is similar in the case where $(x_{k+1},y_{k+1})$ is selected according to~\eqref{eq:MP'1}-\eqref{eq:MP'2}.
\end{proof}

The practical consequence of the above results is that,
for all $k\in\mathbb{N}$, either $|\left(f-I_k(f)\right)(x_{k+1},y_{k+1})|~=~0$ and
$f=I_k(f)$ (Corollary~\ref{corollary_1}), or $I_{k+1}$ can be constructed (Lemma~\ref{lemma_1}) and the greedy procedure can be pursued.

\begin{remark}[Goal-oriented approximation]
In a reduced-basis context, the EIM procedure is typically used
to obtain a separated representation of the coefficients of a parametrized partial differential equation.
As an example, let us consider the right-hand side $f(x,y)$ of a PDE discretized using a Galerkin basis $\{\varphi_1,\ldots,\varphi_N\}$.
Here, $x$ is the parameter and $y$ the space variable. The discretized right-hand side is then $b(x)_j=\int_{\Omega}f(x,y)\varphi_j(y)\,dy$
and using the EIM approximation, it actually can be written as a linear combination of functions of $x$:
$(I_d(b)(x))_j=\int_{\Omega}I_d(f)(x,y)\varphi_j(y)\,dy=\sum_{1\leq l,m\leq d}D_{l,m}f(x,y_m)\int_{\Omega}f(x_l,y)\varphi_j(y)\,dy$.
This separated representation is very important for the overall efficiency of the greedy procedure (see~\cite{Barrault}).
Consider a finite-dimensional subspace of a Hilbert space with basis $\{\varphi_1,..., \varphi_N\}$.
Define $b(x)_{j}=\int_{\Omega}f(x,y)\varphi_j(y) dy$ for all $j\in\{1,\ldots,N\}$, and
its EIM-based approximation $(I_d (b)(x))_{j}=\int_{\Omega}I_d (f)(x,y)\varphi_j(y)dy$. If the EIM approximation $I_d$
has been constructed using~\eqref{eq:MP1}-\eqref{eq:MP2}, the error in the greedy procedure is assessed on the quality of the approximation
of $f(x,y)$. A goal-oriented alternative is to assess the quality of the approximation of $f$ on the object $b$ to be approximated:
\begin{subequations}
\begin{alignat}{1}
x_{k+1}&=\arg\max_{x \in {\mathcal X}}\|\left(b-I_d (b)\right)(x)\|,\\
y_{k+1}&=\arg\max_{y \in {\mathcal Y}} |\left(f-I_k(f)\right)(x_{k+1},y)|,
\end{alignat}
\end{subequations}
where $\|{\cdot}\|$ can be any norm on $\mathbb{R}^N$.
\end{remark}


\section{Extension to rectangular cases}

\subsection{Different number of interpolation points $x$ and $y$} \label{sec:sensor_failure}

Consider a given EIM approximation, where $d$ couples $(x_m, y_m)$, $m\in\calN:=\{1,...,d\}$ have been selected.
Suppose now that for some reason, some points $y_{m_0}$, $m_0\in\calN_0\subsetneq\calN$, must be discarded. For instance, consider that $f(x,y)$ represents
the evaluation of a physical field parametrized by $x$ by a sensor located at a point $y$. Discarding $y_{m_0}$ would model
the failure of the corresponding sensor.
One might still want to use the information provided by the physical field parametrized by $x_{m_0}$ which has been selected together with the sensor
$y_{m_0}$ during the greedy procedure. The motivation is that even if the first selected sensor fails, we can still use the information
from the associated physical field in the construction of the approximate separated representation.

Our idea, inspired from Lemma~\ref{lemma_0} is to choose $D=(F^T)^{\dagger}$, where $\cdot^{\dagger}$ denotes the pseudo-inverse.
The matrices $D$ and $F$ are both rectangular and in $\mathbb{R}^{d\times d_0}$ with $d_0=\mathrm{card}(\calN_0)$.
We recall that if $A\in\mathbb{R}^{d_1\times d_2}$, then $A^{\dagger}\in\mathbb{R}^{d_2\times d_1}$ is the unique matrix verifying
(i) $AA^{\dagger}A=A$, (ii) $A^{\dagger}AA^{\dagger}=A^{\dagger}$, (iii) $(AA^{\dagger})^T=AA^{\dagger}$, (iv) $(A^{\dagger}A)^T=A^{\dagger}A$.
Choosing $D=(F^T)^{\dagger}$, the interpolation properties~\eqref{eq:interpolation_prop1}-\eqref{eq:interpolation_prop2} will not be verified anymore, but
since the goal is to find an approximation formula and not necessarily an interpolation formula, we can define the rectangular EIM approximation by
\begin{equation}
\label{eq:rect_approx}
\sum_{l\in\calN}~\sum_{m\in\calN_0} (F^T)^{\dagger}_{l,m}f(x_l,y)f(x,y_m). 
\end{equation}

As a numerical illustration, fix $\vec{v}=(1~2~3)^T$ and 
consider the function $\mathbb{R}^{3}\times\mathbb{R}\ni(\vec{x},y)\mapsto f(\vec{x},y):=\cos((\vec{v}\cdot\vec{x})y)\in(-1,1)$.
Suppose that $d=8$ couples of points $(x_k,y_k)$, $k\in\calN=\{1,...,d\}$, have been selected by the greedy
procedure according to~\eqref{eq:points_select1}-\eqref{eq:points_select2}.
Consider any pair $i,j\in\calN$, $i\ne j$ and set $\calN_0=\calN\setminus\{i,j\}$.
The relative $\ell^2$-norm error on 1000 sampling points in $(0,1)^3\times(0,1)$ is computed 
by the two following algorithms: (i)~using the classical EIM approximation with the $6$ couples of points
$(x_k,y_k)$, $k\in\calN_0$, and (ii)~using the approximation~\eqref{eq:rect_approx} with the 6 points
$\{x_k\}_{k\in\calN_0}$ in the $x$-dimension and the 8 points $\{y_k\}_{k\in\calN}$ in the $y$-dimension.
The relative $\ell^2$-norm error has been computed for all $\calN_0=\calN\setminus\{i,j\}$,
$i,j\in\calN$ such that $i\ne j$.
In case (i), the maximum, minimum, and mean relative $\ell^2$-norm errors are respectively $6.8\times 10^{-4}$, $3.0\times 10^{-6}$, and $3.6\times 10^{-5}$,
whereas in case (ii), these errors are respectively $2.3\times 10^{-5}$, $7.6\times 10^{-7}$, and
$2.4\times 10^{-6}$. This example illustrates the fact that using the rectangular approximation so as to keep the information from 
$x_i$ and $x_j$, $i,j\in\calN$ such that $i\ne j$, yields a better accuracy than discarding this information and using the square approximation.


\subsection{Application to GEIM~\cite{GEIM}}

Consider a set $\mathcal{F}$ of functions $f$ defined over $\Omega\subset\mathbb{R}$, and a dictionary $\Sigma$ of linear forms $\sigma$ defined
over $\mathcal{F}$. Consider $d$ functions $(f_1, ..., f_d)\in\mathcal{F}^d$ and $d$ linear forms $(\sigma_1,...,,\sigma_d)\in\Sigma^d$
and define the matrix $\hat{F}=\left(\sigma_l(f_m)\right)_{1\leq l,m\in\leq d}$.

\begin{lemma}[Existence for GEIM]
Assume that $\hat{F}$ is invertible. Then, setting $D = \hat{F}^{-T}$, the linear combination
\begin{equation}
\label{eq:decomp_GEIM}
I_d(f)= \sum_{1\leq l,m\leq d} D_{l,m} \sigma_l(f) f_m
\end{equation}
satisfies for all $m\in\{1,\ldots,d\}$,
\begin{subequations}
\begin{alignat}{4}
\sigma_m(f)&=\sigma_m(I_d(f)), \quad &\textnormal{for all }& f \in \mathcal{F},\label{eq:interp_req_GEIM1}\\
\sigma(f_m)&=\sigma(I_d(f_m)), \quad &\textnormal{for all }& \sigma \in \Sigma.\label{eq:interp_req_GEIM2}
\end{alignat}
\end{subequations}
\end{lemma}

\begin{proof}
\eqref{eq:interp_req_GEIM1} implies that
$\sigma_{m_0}(f)=\sigma_{m_0}\left(I_d(f)\right)=\sum_{1\leq l\leq d}\left(\sum_{1\leq m\leq d}D_{l,m}\sigma_{m_0}(f_m)\right) \sigma_l(f)$ and~\eqref{eq:interp_req_GEIM2} implies that
$\sigma(f_{m_0})=\sigma\left(I_d(f_{m_0})\right)=\sum_{1\leq m\leq d}\left(\sum_{1\leq l\leq d}D_{l,m}\sigma_l(f_{m_0})\right) \sigma(f_m)$, meaning that
$D = \hat{F}^{-T}$ can be chosen.
\end{proof}

\begin{lemma}[Stopping criterion for GEIM]
Assume that the matrix $\hat{F}^k:=\left(\sigma_l(f_m)\right)_{1 \le l,m \le k}$ is invertible. If
$|\sigma_{k+1}(f_{k+1}) - \sigma_{k+1}(I_k(f_{k+1}))| \neq 0$, then the matrix $\hat{F}^{k+1}:=(\sigma_l(f_m))_{1 \le l,m \le k+1}$ is invertible.
\end{lemma}

\begin{proof}
Assume that the matrix $\hat{F}^{k+1}$ is not invertible. Then there exists 
$(\lambda_1, \ldots, \lambda_{k+1})$ not all zero and with $\lambda_{k+1}\neq 0$, such that
$\sum_{1\le m\le k} \lambda_m \sigma_l(f_m) = - \lambda_{k+1} \sigma_l(f_{k+1})$, $\forall l \in \{1, \ldots,k+1\}$. We have
$\sum_{1\le l\le k} D^k_{l,m_0} \sum_{1\le m\le k} \lambda_m \sigma_l(f_m)= - \lambda_{k+1} \sum_{1\le l\le k} D^k_{l,m_0} \sigma_l(f_{k+1})$, where
$D^k=(\hat{F}^k)^{-T}$. From the definition of $D^k$ and~\eqref{eq:decomp_GEIM}, we deduce that
$\sum_{1\le m_0\le k} \lambda_{m_0} f_{m_0} = -\lambda_{k+1} I_k(f_{k+1})$. Applying $\sigma_{k+1}$ to both sides of this relation gives
$\sigma_{k+1}(f_{k+1}) = \sigma_{k+1}(I_k(f_{k+1}))$, yielding the desired contradiction.
\end{proof}

The generalization of Theorem~\ref{existence2} to GEIM, for which any norm on $\mathcal{F}$ can be taken, follows immediately.

Let us now return to the setting of sensor failure discussed in Section~\ref{sec:sensor_failure}. This setting nicely fits the GEIM.
Suppose that for all $\sigma\in\Sigma$, there exists a unique $x_{\sigma}\in\Omega$ such that $\sigma(f)=f(x_{\sigma})$. In this setting, the functions $f$ represent the
parametrized physical field depending on $x$, and $\sigma(f)=f(x_{\sigma})$ represents a sensor located at the point $x_{\sigma}$, that measures the value of $f$ at this point.
Discarding the broken sensor $m_0$, we can produce the following rectangular approximation
\begin{equation}
\label{eq:decomp_GEIM_rectangular}
\sum_{l\in\calN_0}~\sum_{m\in\calN} (\hat{F}^T)^{\dagger}_{l,m}\sigma_l(f) f_m,
\end{equation}
which should deliver a more accurate representation of $f$ than if the functions $f_{m_0}$ were also discarded as in the classical (square) EIM context.

\vspace{-0.5cm}
\section*{Acknowledgements}
\vspace{-0.3cm}
The authors would like to thank Y. Maday (Laboratoire Jacques-Louis Lions) for fruitful discussions.
The work of T. Leli\`evre is supported by the European Research Council under the European Union's Seventh Framework Programme
(FP/2007-2013) / ERC Grant Agreement number 614492.

\end{document}